\documentclass[reqno]{amsart}

\pdfoutput=1

\usepackage[alphabetic]{amsrefs}
\usepackage{amsmath}
\usepackage{amsfonts}
\usepackage{amssymb}
\usepackage{enumerate}
\usepackage{amstext}
\usepackage{amsbsy}
\usepackage{amsopn}
\usepackage{bbm,amsthm}
\usepackage{amscd}
\usepackage[pdftex]{color}
\usepackage[pdftex]{graphicx}
\usepackage[all]{xy}
\usepackage{mathtools}
\usepackage{todonotes}
\usepackage{soul}

\usepackage{tikz}
\usetikzlibrary{intersections,positioning,patterns,arrows,decorations.markings}

\usepackage{todonotes}

\usepackage{amsxtra}
\usepackage{upref}
\usepackage{epstopdf}
\usepackage{graphicx,color}
\usepackage{hyperref}
\usepackage{longtable}
\usepackage[francais, english]{babel}

\makeatletter
\newsavebox{\@brx}
\newcommand{\llangle}[1][]{\savebox{\@brx}{\(\m@th{#1\langle}\)}%
  \mathopen{\copy\@brx\kern-0.5\wd\@brx\usebox{\@brx}}}
\newcommand{\rrangle}[1][]{\savebox{\@brx}{\(\m@th{#1\rangle}\)}%
  \mathclose{\copy\@brx\kern-0.5\wd\@brx\usebox{\@brx}}}
\makeatother

\DeclareFontFamily{OML}{rsfs}{\skewchar\font'177}
\DeclareFontShape{OML}{rsfs}{m}{n}{ <5> <6> rsfs5 <7> <8> <9> rsfs7
  <10> <10.95> <12> <14.4> <17.28> <20.74> <24.88> rsfs10 }{}
\DeclareMathAlphabet{\mathfs}{OML}{rsfs}{m}{n}

\newcommand{\interior}{\operatorname{int}}

\newtheorem{theorem}{Theorem}

\newtheorem{ltheorem}{Theorem} 
\newtheorem{lemma}[theorem]{Lemma}
\newtheorem{proposition}[theorem]{Proposition}
\newtheorem{corollary}[theorem]{Corollary}

\newtheorem*{claim}{Claim}
\theoremstyle{definition}

\theoremstyle{remark}
\newtheorem{remark}[theorem]{\bf Remark}

\numberwithin{equation}{section}
\numberwithin{theorem}{section}

\newcommand{\intav}[1]{\mathchoice {\mathop{\vrule width 6pt height 3 pt depth  -2.5pt
\kern -8pt \intop}\nolimits_{\kern -6pt#1}} {\mathop{\vrule width
5pt height 3  pt depth -2.6pt \kern -6pt \intop}\nolimits_{#1}}
{\mathop{\vrule width 5pt height 3 pt depth -2.6pt \kern -6pt
\intop}\nolimits_{#1}} {\mathop{\vrule width 5pt height 3 pt depth
-2.6pt \kern -6pt \intop}\nolimits_{#1}}}

\newcommand{\intavl}[1]{\mathchoice {\mathop{\vrule width 6pt height 3 pt depth  -2.5pt
\kern -8pt \intop}\limits_{\kern -6pt#1}} {\mathop{\vrule width 5pt
height 3  pt depth -2.6pt \kern -6pt \intop}\nolimits_{#1}}
{\mathop{\vrule width 5pt height 3 pt depth -2.6pt \kern -6pt
\intop}\nolimits_{#1}} {\mathop{\vrule width 5pt height 3 pt depth
-2.6pt \kern -6pt \intop}\nolimits_{#1}}}

\newcommand{\vertiii}[1]{{\left\vert\kern-0.2ex\left\vert\kern-0.2ex\left\vert #1 
    \right\vert\kern-0.2ex\right\vert\kern-0.2ex\right\vert}}

\newcommand{\un}{\underline}

\newcommand{\ve}{\varepsilon}

\newcommand{\vf}{\varphi}

\newcommand{\R}{\mathbb{R}}

\begin{document}

\title[Unique MME for geodesic flows on surfaces]{Uniqueness of the measure of maximal\\ 
entropy for geodesic flows on surfaces}

\author{Yuri Lima,  Davi Obata and Mauricio Poletti}

\address{Instituto de Matemática e Estatística, Universidade de São Paulo, Rua do Matão, 1010, Cidade Universitária, 05508-090. São Paulo -- SP, Brazil}
\email{yurilima@gmail.com}
\address{Department of Mathematics, Brigham Young University, Provo, Utah, 84602, USA}
\email{davi.obata@mathematics.byu.edu}
\address{Departamento de Matem\'atica, Universidade Federal do Cear\'a (UFC), Campus do Pici,
Bloco 914, CEP 60455-760. Fortaleza -- CE, Brasil}
\email{mpoletti@mat.ufc.br}

\date{\today}

\begin{abstract}
We prove that if a geodesic flow on a closed orientable $C^\infty$ surface is transitive and
has positive topological entropy, then it has a unique measure of maximal entropy.
This covers all previous results
of the literature on the uniqueness of the measure of maximal entropy in this context,
as well as it applies to new examples such as the ones constructed by Donnay and Burns-Donnay.
We also prove that, in the above context, there is at most one SRB measure.

\end{abstract}

\maketitle

\section{Introduction}

Entropy quantifies the complexity of a dynamical system. Among its various notions,
two of particular interest in dynamical systems are the {\em topological entropy} and the {\em metric entropy}
(also known as the Kolmogorov-Sina{\u\i} entropy). The celebrated {\em variational principle} asserts that,
under some regularity assumptions, the topological entropy equals the supremum of the metric entropies
over all invariant probability measures. An invariant probability measure that attains this supremum is called a \emph{measure of maximal entropy}. Understanding  conditions that ensure its existence, finiteness,
and uniqueness has been a theme of intense research over the past several decades.
In this paper, we provide sufficient conditions for the uniqueness of the measure of maximal
entropy for geodesic flows on surfaces.
Given a closed orientable $C^\infty$ Riemannian surface $(S,g)$, where $g$ is a Riemannian metric, let $M=T^1S$ denote its unit tangent bundle
and $\varphi:M\to M$ be the geodesic flow on $(S,g)$. Let $h_{\rm top}(\vf)$ denote the topological entropy
of $\vf$. We prove the following.

\begin{ltheorem}\label{thm.geodesic}
If $\varphi$ is transitive and $h_{\rm top}(\vf)>0$, then it has a 
unique measure of maximal entropy $\mu$, which is Bernoulli. If additionally there is a
$\varphi$--invariant ergodic probability measure with positive entropy and full support,
then $\mu$ has full support.
\end{ltheorem}

Theorem \ref{thm.geodesic} makes no assumption on the curvature of $(S,g)$, hence
it covers all previous results on the uniqueness of the measure of maximal entropy
for geodesic flows on surfaces, and it also applies to surfaces with
conjugate points such as those constructed by Donnay \cite{Donnay-I,Donnay-II}, Burns and Marlies \cite{BurnsMarlies89}
and by Burns and Donnay \cite{Burns-Donnay}, see Figure \ref{fig-Donnay}.

 \begin{figure}[hbt!]\label{fig-Donnay}
\centering
\def\svgwidth{4.5cm}
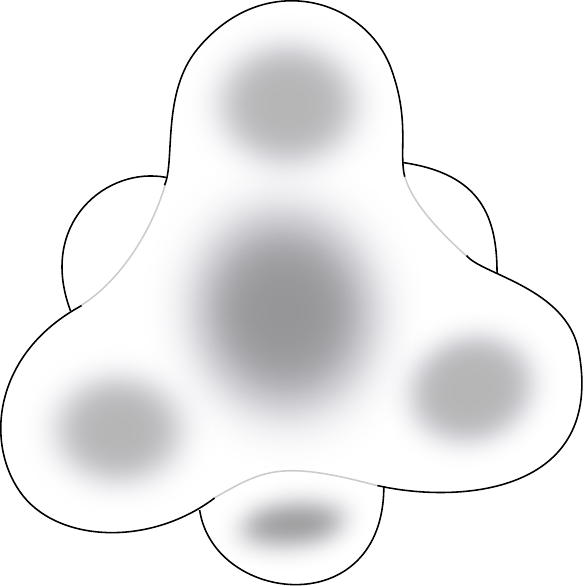\caption{A metric on the sphere $\mathbb S^2$ whose 
geodesic flow is ergodic and hyperbolic for
the Liouville measure \cite{Burns-Donnay}.}
\end{figure}

\begin{corollary}
Let $\vf$ be the geodesic flow on a closed orientable $C^\infty$ Riemannian surface $(S,g)$, 
and assume that the Liouville measure is ergodic and hyperbolic. Then $\vf$ has
a unique measure of maximal entropy, which is Bernoulli and fully supported.
\end{corollary}

In particular, Donnay's examples
have a unique measure of maximal entropy, which is Bernoulli and fully supported. 

\begin{proof}
If ${\rm vol}$ is the Liouville measure, the Pesin entropy formula gives $h_{\rm vol}(\vf)>0$ 
and so $\vf$ has positive topological entropy. Since ${\rm vol}$ is ergodic, $\vf$ is transitive.
\end{proof}

Up to our knowledge, understanding which geodesic flows on surfaces are transitive and have positive
topological entropy is a general open question, with some partial results. For instance,
a $C^\infty$ generic metric on the sphere $\mathbb S^2$ has positive topological
entropy \cite{Knieper-Weiss-Generic-Positive-entropy}; a metric on $\mathbb S^2$ with
positive curvature with transitive geodesic flow has positive topological entropy \cite{Schroder},
although it is not know whether such metric exists. Due to KAM theory,  transitivity is not a generic property in this context.

The proof of Theorem \ref{thm.geodesic} uses two main ingredients that have been developed
recently by different groups. The first deals with codings of homoclinic classes of hyperbolic
measures, following Buzzi, Crovisier and Sarig \cite{BCS} and Buzzi, Crovisier and Lima \cite{BCL}.
The second deals with Birkhoff sections for geodesic flows on surfaces, 
following the works of Contreras, Mazzucchelli \cite{CM22}, Contreras et al \cite{CKMS25} and
more recently Alves and Mazzucchelli \cite{AM25}, which imply the existence of a Birkhoff section for
every geodesic flow on a surface $(S,g)$.

We indeed prove a more general result for (not necessarily geodesic) flows without fixed points on
three dimensional manifolds, for the more general class of measures known as equilibrium states.
This is our Theorem \ref{maintheorem} below, which is our main result. 
Theorem \ref{thm.geodesic} is then a direct consequence of Theorem \ref{maintheorem} and \cite{AM25}.
To explain our main theorem, we need to introduce some notation.

Let $\vf$ be a flow on a closed
three dimensional $C^\infty$ Riemannian manifold $M$.
Given a continuous potential $\psi:M\to \mathbb{R}$, the {\em topological pressure}
of $\psi$ is 
$$
P_{\rm top}(\psi)=\sup\left\{h_\mu(\varphi)+\int \psi d\mu:\mu\text{ is $\vf$--invariant probability measure}\right\},
$$
where $h_\mu(\varphi)$ denotes the metric entropy of $\vf$ with respect to $\mu$.
A $\vf$--invariant probability measure $\mu$ such that $P_{\rm top}(\psi)=h_\mu(\varphi)+\int \psi d\mu$
is called an {\em equilibrium state of $\psi$}. In particular, measures of maximal entropy are
equilibrium states of the zero potential $\psi\equiv 0$ and $h_{\rm top}(\vf)=P_{\rm top}(0)$.

\medskip
\noindent
{\sc Positive entropy potential:} We call $\psi$ a {\em positive entropy potential}
if every ergodic equilibrium state of $\psi$ has positive metric entropy.

\medskip
If $\varphi$ has positive topological entropy then $\psi\equiv 0$ is a positive entropy potential;
more generally, every $\psi$ with $\sup \psi -\inf \psi< h_{\rm top}(\vf)$ is a positive entropy potential.

\begin{ltheorem}\label{maintheorem}
Let $\vf:M\to M$ be a $C^\infty$ flow with positive speed on a closed
three dimensional $C^\infty$ Riemannian manifold $M$, and assume $\vf$ is transitive
and has a Birkhoff section. Then every H\"older continuous positive entropy potential has a unique
equilibrium state $\mu$, which is Bernoulli times a rotation (if $\vf$ is a Reeb flow, then $\mu$ is Bernoulli).

If additionally there is a $\varphi$--invariant ergodic probability measure 
with positive entropy and full support, then $\mu$ has full support.
\end{ltheorem}

This more general result applies to the transitive Reeb flows inside the open and dense set
of Reeb flows that have a Birkhoff section and positive topological entropy
that was shown to exist by Colin et al \cite{CDHR24} (up to our knowledge,
there is no known example of a Reeb flow in dimension three without a Birkhoff section).
The Bernoulli property of $\mu$ was established in \cite{LLS}, including the case when $\vf$ is a Reeb flow.

\begin{remark}\label{rmk-admissible}
Theorem \ref{maintheorem} actually holds under more general conditions,
when $\psi$ is  a positive entropy potential such that its lift under the symbolic coding $\pi_r:\Sigma_r\to M$
constructed in \cite{BCL} is Hölder continuous in the Bowen-Walters metric, see also \cite{LMN}.
This allows to apply the result to the geometric potential,
which is not known to be Hölder in the manifold metric but it essentially is in the symbolic 
metric\footnote{
Actually, the lift of the geometric potential to the symbolic space is cohomologous
to a Hölder continuous function, see \cite[Section 8]{LLS}.}.
We will use this important fact to prove Theorem \ref {theoremSRB} below.
\end{remark}

\begin{remark}
Theorem \ref{maintheorem} also has a version for $C^r$ flows (assuming a ``high entropy'' regime), which
{\em does not} provide the existence of an equilibrium state, but gives uniqueness when it exists,
see Theorem \ref{finitereg}.
\end{remark}

Under $C^\infty$ regularity, the existence of a measure of maximal entropy follows from the work of Newhouse \cite{Newhouse-Entropy} (indeed, this holds for every continuous potential).
This paper deals with uniqueness. The question on the (existence and) uniqueness of measures of maximal entropy
for maps and flows, and in particular for geodesic flows, is a classical topic on dynamical 
systems and ergodic theory, with a long list of contributions since the early 1970s. 
In the past fifteen years, new developments have been obtained, under 
some weak forms of hyperbolicity, such as ``no conjugate points'' or ``no focal points''. 
The list of contributions below exemplifies (but is not exhaustive) the 
contributions in this direction:
\begin{enumerate}[$\circ$]
\item Uniqueness of measure of maximal entropy for transitive uniformly
hyperbolic flows \cite{Bowen-maximizing}.
\item Uniqueness of measure of maximal entropy for geodesic flows on
rank one manifolds \cite{Knieper-Rank-One-Entropy}.
\item At most countably many measures of maximal entropy for three dimensional flows with positive speed
and positive topological entropy \cite{LLS}.
\item Uniqueness of equilibrium states for a wide class of potentials,
for geodesic flows on rank one manifolds \cite{BCFT-2018}. Similar results were extended
for geodesic flows on surfaces without focal points \cite{Chen-Kao-Park}.
\item Uniqueness of measure of maximal entropy for geodesic flows on surfaces
without focal points \cite{GR19} and without conjugate points and continuous Green bundles \cite{GR23}.
\item Uniqueness of measure of maximal entropy for geodesic flows on surfaces
without conjugate points \cite{CKW21}, see also the alternative proof \cite{Mamani24}. 
\item Finiteness of measures of maximal entropy for three dimensional flows with positive speed
and positive topological entropy \cite{Yuntao25}.
\end{enumerate}
There are also recent results in high dimension \cite{MR23,Wu24,Knieper2025}.
The techniques of the cited above works vary, including symbolic dynamics, Patterson-Sullivan measures,
expansive factors, and weak notions of specification. 
Theorem \ref{thm.geodesic} includes all the above results on the uniqueness of the measure
of maximal entropy for geodesic flows on surfaces. We remark that many of these works
obtain additional properties on the unique measure of maximal entropy, such as the
equidistribution of periodic orbits, which we do not address here.
As already mentioned, our results also apply to previously unknown contexts such as the surfaces
constructed in \cite{Donnay-I,Donnay-II,Burns-Donnay}.

Assuming a similar context to Theorem \ref{maintheorem} (but not requiring $C^\infty$ regularity),
we also prove that there exists at most one SRB measure, as we now explain. 
Let $\vf$ be a $C^r$ flow ($r>1$) with positive speed on a closed
three dimensional $C^\infty$ Riemannian manifold $M$, and let
$\lambda^+(x)=\limsup\limits_{t\to \infty}\frac{1}{t}\log \|d\varphi^t(x)\|$
and $\lambda^-(x)=\liminf\limits_{t\to \infty}\frac{1}{t}\log m(d\varphi^t(x))$
denote the Lyapunov exponents of $\vf$ at $x$.\footnote{Given a 
linear transformation $A$, we denote its conorm by $m(A)$.}
We say that a $\vf$--invariant probability measure $\mu$ is hyperbolic
of saddle type if $\lambda^-(x) < 0 < \lambda^+(x)$ for $\mu$--a.e. $x$.
If additionally $h_\mu(\varphi)=\int_M \lambda^+(x)d\mu$, then $\mu$ is called
an {\em SRB measure} for $\vf$.

\begin{ltheorem}\label{theoremSRB}
Let $\vf:M\to M$ be a $C^r$ flow $(r>1)$ with positive speed on a closed
three dimensional $C^\infty$ Riemannian manifold $M$. If $\vf$ is transitive
and has a Birkhoff section, then it admits at most one SRB measure. 
\end{ltheorem}

This is a version for three dimensional flows of the theorem for surface diffeomorphisms of Hertz et al \cite{HHTU}.
Partial results for flows in any dimension have been recently obtained in \cite{JEP25}.
By the Pesin entropy formula, any absolutely continuous hyperbolic invariant measure is SRB,
hence Theorem \ref{theoremSRB} provides the following ergodicity (hence Bernoulliness)
criterion for the Liouville measure.

\begin{ltheorem}\label{theorem-NUH}
Let $\vf$ be the geodesic flow on a closed orientable $C^\infty$ Riemannian surface,
and assume $\vf$ is transitive. If the Liouville measure is hyperbolic,
then it is Bernoulli.
\end{ltheorem}

\medskip
\noindent
{\em{Acknowledgements.}} We are thankful to Marcelo Alves, for explaining
the construction of Birkhoff sections.  We are also thankful to Ygor de Jesus and Rafael Ruggiero
for useful conversations. YL was supported by 
CNPq/MCTI/FNDCT project 406750/2021-1,
FUNCAP grant UNI-0210-00288.01.00/23, Instituto Serrapilheira grant
``Jangada Din\^{a}mica: Impulsionando Sistemas Din\^{a}micos na Regi\~{a}o Nordeste'', and
FAPESP grant number 2025/11400-7.
DO was partially supported by the National Science Foundation under Grant No. \ DMS-2349380. 
YL and DO were also supported by the Jacob Palis ABC-Fulbright Award. MP was partially supported by CAPES-Finance Code 001, Instituto Serrapilheira grant number Serra-R-2211-41879 and FUNCAP grant AJC 06/2022.
\section{Preliminaries}

Thoughout this paper, $M$ is a closed $C^\infty$ Riemannian three dimensional manifold,
$X$ is a $C^\infty$ vector field on $M$ such that $X(p)\neq 0$ for all $p\in M$, and 
$\vf=\{\vf^t\}_{t\in\R}:M\to M$ is the flow generated by $X$.
Given $x\in M$ and an interval $I\subset\R$, we write $\vf^I(x)=\{\vf^t(x):t\in I\}$.

\subsection{Birkhoff sections}

The main geometrical tool we will use is the following.

\medskip
\noindent
{\sc Birkhoff section:} A {\em Birkhoff section} of $\vf$ is a (possibly disconnected) immersed 
compact surface $\Sigma \looparrowright M$ such that:
\begin{enumerate}[(B1)]
\item $\partial \Sigma$ consists of finitely many periodic orbits of $\varphi$,
\item ${\rm int}(\Sigma)$ is embedded in $M\setminus \partial \Sigma$
and everywhere transverse to $X$, and
\item $\exists\, T>0$ such that for every $x\in M$ the orbit segment
$\varphi^{[0,T]}(x)$ intersects $\Sigma$.
\end{enumerate}

\medskip
\noindent
{\sc Poincaré return map:} Given a Birkhoff section $\Sigma$, we define
its {\em Poincaré return time map} $\tau:{\rm int}(\Sigma)\to (0,T]$ by
$$
\tau(x)=\min\{t>0:\vf^t(x)\in \Sigma\}
$$
and its {\em Poincaré return map} $f:{\rm int}(\Sigma)\to{\rm int}(\Sigma)$
by $f(x)=\vf^{\tau(x)}(x)$. 

\medskip
By definition, all flow trajetories but the finitely many 
periodic ones composing $\partial\Sigma$ intersect ${\rm int}(\Sigma)$ and hence
can be analyzed via $f$. The fact that $\tau>0$ follows from (B2).

\begin{remark}
Although it will not be used here, we point that Alves and Mazzucchelli proved
that if $\vf$ is the geodesic flow on a surface, then
there are Birkhoff sections with $\inf(\tau)>0$, see \cite[Remark 6.2]{AM25}. 
This allows to obtain a bijection between $\vf$--invariant ergodic probability measures
not supported on $\partial\Sigma$ and $f$--invariant ergodic probability measures
on ${\rm int}(\Sigma)$, see e.g. \cite[Lemma 2.3]{GKL25}.
\end{remark}

\begin{lemma}
The map $\tau$ is $C^\infty$. In particular, $f$ is a $C^\infty$ diffeomorphism.
\end{lemma}

\begin{proof}
Fix $x\in {\rm int}(\Sigma)$ and let $t=\tau(x)$. By (B2), we can take $U\ni x$ and $V\ni \vf^{t}(x)$ small neighborhoods of 
${\rm int}(\Sigma)$ and $\ve>0$ small such that 
$\widetilde U=\vf^\ve U$ and $\widetilde V=\vf^{-2\ve}V$ are sections transverse to $X$ and
$\vf^{(0,\ve]}U\cap \Sigma=\vf^{[-2\ve,0)}V\cap \Sigma=\emptyset$. Since $\vf^{[\ve,t-\ve]}(x)$ is a compact
set disjoint from $\Sigma$, by the tubular neighborhood theorem there is $W\subset \widetilde U$
such that $\vf^{[0,t-2\ve]}W\cap \Sigma=\emptyset$ and $\vf^{t-2\ve}W\subset \vf^{[-2\ve,0)}V$.
Then on a neighborhood of $x$ we have
$\tau=\ve+\widetilde\tau$, where $\widetilde\tau$ is the $C^\infty$ flow time from $W$
to $V$.
\end{proof}

We will consider periodic points of $f$. If $\mathcal O$ is a hyperbolic periodic orbit
of $\varphi$ that is different from those composing
$\partial\Sigma$, then every $x\in \mathcal O\cap \Sigma$ is a hyperbolic periodic point for $f$.

\begin{lemma}\label{lemma-f-transitive}
If $\varphi$ is transitive then $f$ is transitive.
\end{lemma}

\begin{proof}
If $\vf$ is a geodesic flow on a surface, then by \cite{AM25} we can assume that $\ve_0:=\inf(\tau)>0$
and proceeds as follows. Take $x\in M$
with dense $\vf$--orbit on $M$. Since $x\not\in\partial\Sigma$, we can assume
that $x\in{\rm int}(\Sigma)$. Let $\mathcal T=\{0=t_0<t_1<t_2<\cdots\}$ be the discrete infinite set of times where
$\vf^{t_i}(x)\in{\rm int}(\Sigma)$, so that $f^n(x)=\vf^{t_n}(x)$.
Given $U$ open in ${\rm int}(\Sigma)$, condition (B2) implies that $V=\vf^{(0,\ve_0)}U$ is
open in $M$. By assumption, there is $t>0$ such that $\vf^t(x)\in V$ and so 
$\vf^{[t-\ve_0,t+\ve_0]}(x)\cap U\neq\emptyset$. Then there is $t_n\in[t-\ve_0,t+\ve_0]$
such that $\vf^{t_n}(x)\in U$. 

Now we give a proof that does not assume $\ve_0>0$. We prove that for every $x\in M$ and
$t>0$ the set $\mathcal T=\{s\in[0,t]:\vf^s(x)\in {\rm int}(\Sigma)\}$ is finite. 
For every $s\in\mathcal T$, condition (B2) allows to construct a flow box
$V_s=\vf^{(s-\delta(s),s+\delta(s))}U$ with $U\subset {\rm int}(\Sigma)$ such that
$\vf^{[s-\delta(s),s+\delta(s)]}(y)\cap U=\{y\}$
for every $y\in U$. These flow boxes form an open cover of the compact set
$\vf^{\mathcal T}(x)=\vf^{[0,t]}(x)\cap \Sigma$, hence there is a finite cover 
$V_{s_1},\ldots,V_{s_n}$. Therefore $|s-s'|>\min\{\delta(s_1),\ldots,\delta(s_n)\}$
for every $s,s'\in\mathcal T$, thus proving the claim.
The conclusion of the lemma is now direct: letting $x\in {\rm int}(\Sigma)$ with dense $\vf$--orbit and
proceeding as in the previous paragraph,
if $U\subset{\rm int}(\Sigma)$ is open then there is $t>0$ such that $\vf^t(x)\in U$, and 
necessarily $\vf^t(x)=f^n(x)$ for some $n>0$. 
\end{proof}

\subsection{Homoclinic relation of measures}

A $\vf$--invariant probability measure $\mu$ on $M$ is {\em hyperbolic}
if it has two non-zero Lyapunov exponents (the third exponent, in the flow direction, is zero).
When this happens, $\mu$--a.e.
$x\in M$ has a {\em strong stable manifold} 
$$
W^{ss}(x)=\left\{y\in M:\limsup_{t\to+\infty}\tfrac{1}{t}\log d(\vf^t(x),\vf^t(y))<0\right\}
$$
and a {\em stable manifold}
$$
W^{s}(x)=\bigcup_{t\in\R}\vf^t[W^{ss}(x)].
$$
We define similarly $W^{uu}(x)$ and $W^{u}(x)$
the {\em strong unstable} and {\em unstable} manifolds of $x$. 

Given a hyperbolic periodic orbit $\mathcal O$, we let
$W^{s/u}(\mathcal O)=W^{s/u}(x)$ denote
the stable/unstable manifold of $\mathcal O$, for any $x\in\mathcal O$. 
Recall that two hyperbolic periodic orbits $\mathcal O,\mathcal O'$ are
{\em homoclinically related} if $W^s(\mathcal O)\pitchfork W^u(\mathcal O')\neq\emptyset$
and $W^s(\mathcal O')\pitchfork W^u(\mathcal O)\neq\emptyset$.
When this happens, we write $\mathcal O\stackrel{h}{\sim}\mathcal O'$.

\medskip
\noindent
{\sc Homoclinic class of hyperbolic periodic orbit:} The {\em homoclinic class}
of a hyperbolic periodic orbit $\mathcal O$ is the set
$$
{\rm HC}(\mathcal O)=\overline{\{\mathcal O': \mathcal O\text{ and }\mathcal O'\text{ are homoclinically related}\}}=\overline{W^u(\mathcal O)\pitchfork W^s(\mathcal O)}.
$$  

\medskip
The second equality above is a well-known fact that follows from the classical inclination lemma and the Birkhoff-Smale Theorem. 

\medskip
\noindent
{\sc Homoclinic relation of measures \cite{BCS,BCL}:} 
We say that two ergodic hyperbolic measures $\mu,\nu$ are
\emph{homoclinically related} if for $\mu$--a.e. $x$ and $\nu$--a.e. $y$
there exist transverse intersections $W^{s}(x)\pitchfork W^{u}(y)\ne\emptyset$
and  $W^{u}(x)\pitchfork W^{s}(y)\ne\emptyset$. When this happens,
we write $\mu\stackrel{h}{\sim}\nu$.

\medskip
The homoclinic relation is an equivalence relation among ergodic hyperbolic measures,
see \cite[Prop. 10.1]{BCL}. The following theorem, whose proof uses symbolic dynamics
for the homoclinic class of a measure, is the second main tool of our proof.
It appears in \cite[Corollary 1.2]{BCL} for measures of maximal entropy in dimension three,
and in \cite[Corollary 1.2]{LMN} for more general potentials.
Below, we state it in dimension three for Hölder continuous potentials. 
 
\begin{theorem}[\cite{BCL,LMN}]\label{thm-BCL}
 Let $\vf:M\to M$ be a $C^\infty$ flow with positive
speed on a three dimensional closed $C^\infty$ Riemannian manifold $M$,
and let $\psi:M\to \mathbb{R}$ be Hölder continuous. If $\mu$ is a hyperbolic ergodic measure,
then $\psi$ has at most one equilibrium state that is homoclinically related to $\mu$.
In particular, if $\vf$ has positive topological entropy then $\psi$ has at most one measure of maximal entropy homoclinically related to $\mu$.
\end{theorem}

We also consider the following definition. 

\medskip
\noindent
{\sc Topological homoclinic class:} The {\em topological homoclinic class}
of $\mu$ is the subset of $M$ defined as
$$
{\rm HC}(\mu)=\overline{\bigcup\{{\rm supp}(\nu):\nu\stackrel{h}{\sim}\mu\}}.
$$ 
This set equals ${\rm HC}(\mu)={\rm HC}(\mathcal O)$ for any hyperbolic period orbit
$\mathcal O\stackrel{h}{\sim}\mu$.

\medskip
The characterization ${\rm HC}(\mu)={\rm HC}(\mathcal O)$ above is a version for flows of \cite[Corollary 2.14]{BCS}.

\section{Proof of Theorem \ref{maintheorem}}

By the Ruelle inequality and Theorem \ref{thm-BCL}, it is enough to prove that every two
ergodic measures with positive entropy are homoclinically related. The idea is to apply a two dimensional topological argument
of \cite{BCS} to the first return map $f$ of a Birkhoff section. That requires
a small discussion on stable/unstable curves for $f$. For general $x\in {\rm int}(\Sigma)$,
relating the stable/unstable curves of $x$ for $\vf$ and for $f$ is difficult, since the
trajectory of $x$ might approach $\partial\Sigma$. But if $\mathcal O$ is a periodic orbit for $\vf$ 
not intersecting $\partial \Sigma$,
then it is at a positive distance to $\partial\Sigma$, so that 
the flow projection to ${\rm int}(\Sigma)$ of a small curve $\gamma\subset W^{s}(\mathcal O)$
defines a curve $\widetilde\gamma$ such that $f^n(\widetilde\gamma)$ converges to $\mathcal O$
as $n\to\infty$. For simplicity, we will call such projection
a {\em stable curve of $x$} for every $x\in\mathcal O\cap{\rm int}(\Sigma)$, and similarly for unstable curves.

\medskip
\noindent
{\sc $su$--quadrilateral:} An $su$--quadrilateral of $f$ is a compact connected set $Q\subset {\rm int}(\Sigma)$
such that:
\begin{enumerate}[(Q1)]
\item $Q=\overline{\interior(Q)}$, and
\item there is a periodic point $x$ of $f$ such that
$\partial Q$ is a curve composed by two disjoint unstable curves of $x$
and two disjoint stable curves of $x$.
\end{enumerate} 
We write $\partial^{u/s}Q$ for the the two unstable/stable curves of $\partial Q$.

\begin{proposition}\label{prop-periodic-point}
If $\mu$ is a $\vf$--invariant ergodic probability measure with positive entropy,
then one can find $\mathcal O$ a hyperbolic periodic orbit homoclinically related to $\mu$ that
is disjoint from $\partial\Sigma$ and a quadrilateral $Q$ defined by $\mathcal O$.
\end{proposition}

\begin{proof}
The proof requires applying the inclination lemma stated in the proof of \cite[Prop. 10.1]{BCL}:\\

\noindent
{\bf Inclination lemma.}
\emph{For any hyperbolic measure $\mu$,
there is a set $Y\subset M$ of full $\mu$--measure satisfying the following:
if $x\in Y$, $D\subset W^u(x)$ is a two-dimensional disc 
and $\Delta$ is a two-dimensional disc tangent to $X$ having a transverse intersection point with $W^s(x)$,
then there are discs $\Delta_k\subset \varphi_{(k,+\infty)}(\Delta)$ which converge to $D$ in the $C^1$ topology.}\\

Let $\mu$ be a $\vf$--invariant ergodic probability measure with positive entropy.
We start by showing that there is a periodic orbit $\mathcal O$ disjoint from $\partial\Sigma$ that is homoclinically
related to $\mu$. One way to prove this is to apply a flow version of the Katok horseshoe theorem \cite{Katok-ICM}.
Another is to use the symbolic dynamics of Lima and Sarig \cite{Lima-Sarig}: there is a
Hölder continuous map $\pi_r:\Sigma_r\to M$ defined on a topological Markov flow $\Sigma_r$
such that $\mu[\pi_r(\Sigma_r)]=1$. By \cite{BCL}, $\Sigma_r$ can be chosen irreducible.
Given $x=\pi_r(\un R,0)$, the irreducibility allows to approximate $\un R$ by periodic points
$\un R^{(n)}$, and so $p_n=\pi_r(\un R^{(n)},0)$ converges to $x$. Choosing $x$ on the support of $\mu$,
for a positive $\mu$--measure set $A$ in a neighborhood of $x$ the invariant manifolds $W^{s/u}(y)$ are large
for all $y\in A$ and so they intersect transversally $W^{u/s}(p_n)$. Since $\partial\Sigma$
is composed of finitely many periodic orbits, we can take $\mathcal O$ disjoint from $\partial\Sigma$.

Fix one such periodic orbit $\mathcal O$. Fix $x\in Y$ such that
$$
W^s(\mathcal O)\pitchfork W^u(x)\neq\emptyset \text{ and } W^u(\mathcal O)\pitchfork W^s(x)\neq\emptyset.
$$
Flowing under $\vf$, we can assume that $x\in{\rm int}(\Sigma)$. Consider a small open ball $B$ around $x$
that does not intersect $\partial\Sigma$. Let $B^{u/s}$ be the connected component
of $W^{u/s}(x)\cap B$ that contains $x$, and $\gamma^{u/s}=B^{u/s}\cap{\rm int}(\Sigma)$.
See Figure \ref{figure-quadrilateral}.

By the inclination lemma,
$W^u(\mathcal O)$ accumulates
on $B^u$, hence there are disjoint sets
$B^u_1,B^u_2\subset W^u(\mathcal O)$ which are $C^1$--close to $B^u$.
Their flow projection to ${\rm int}(\Sigma)$ are two unstable curves
 $\gamma^u_1,\gamma^u_2$ associated to $\mathcal O$, which are 
 $C^1$--close to $\gamma^u$.
\begin{figure}[hbt!]
\centering
\def\svgwidth{12cm}
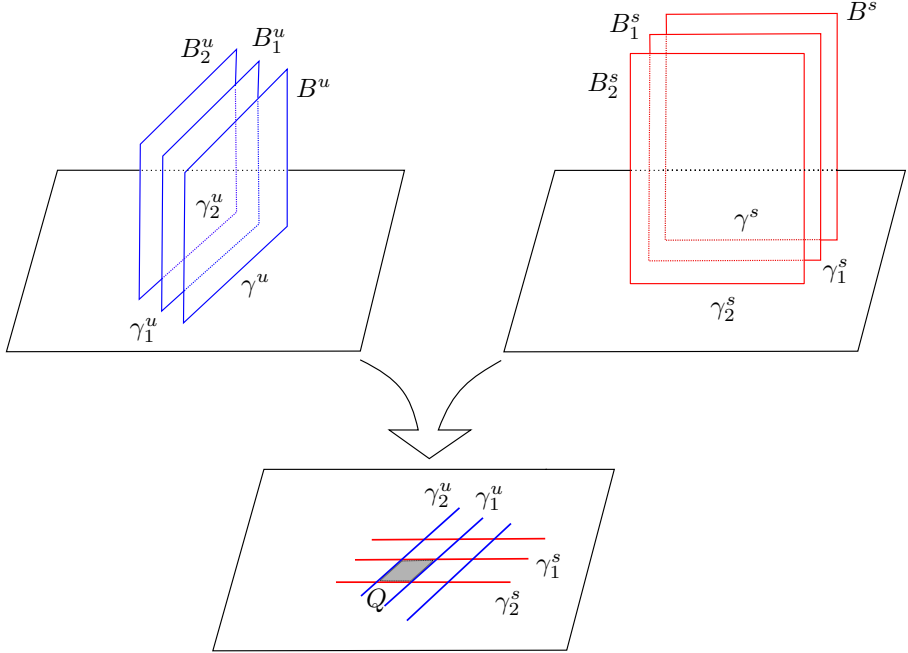\caption{Construction of $su$--quadrilateral.}\label{figure-quadrilateral}
\end{figure}
The same argument applied to the stable direction provides 
two stable curves
$\gamma^s_1,\gamma^s_2$ associated to $\mathcal O$, both 
$C^1$--close to $\gamma^s$. The four curves $\gamma^u_1,\gamma^u_2,\gamma^s_1,\gamma^s_2$
define a $su$--quadrilateral $Q\subset B\cap{\rm int}(\Sigma)$ associated to $\mathcal O$. 
\end{proof}

\begin{proposition}\label{prop-topological-intersection}
Let $\mathcal O_1,\mathcal O_2$ be periodic orbits of $f$ and $Q_1,Q_2$ be $su$--quadrilaterals 
defined by $\mathcal O_1,\mathcal O_2$ such that $Q_1\cap \mathcal O_2=Q_2\cap\mathcal O_1=\emptyset$. 
There exists $n>0$ such that $f^n(\partial^uQ_1)$ intersects $\interior(Q_2)$ and
$\interior(\Sigma)\setminus Q_2$.
An analogous statement holds for $\partial^s Q_1$.
\end{proposition}

Above, the interior of $Q_2$ is considered in the relative topology of ${\rm int}(\Sigma)$.

\begin{proof}
We follow ipsis literis the argument of Step 2 in the proof of \cite[Theorem 6.4]{BCS}.
We have decided to include it to stress the purely topological nature of the argument, regardless
$f$ being defined in a non-compact domain.
By Lemma \ref{lemma-f-transitive}, there is $n>0$ as large as we want such that $f^n(Q_1)\cap Q_2\neq\emptyset$.
We know that $f^n(\partial^sQ_1)$ converges to $\mathcal O_1$ as $n\to\infty$. Since
$Q_2\cap\mathcal O_1=\emptyset$, 
we can assume that $f^n(\partial^sQ_1)\subset {\rm int}(\Sigma)\setminus Q_2$. Since
the endpoints of $f^n(\partial^uQ_1)$ are the same of $f^n(\partial^sQ_1)$, it follows
that $f^n(\partial^uQ_1)$ intersects $\interior(\Sigma)\setminus Q_2$.
Assume that $f^n(\partial^uQ_1)$ does not intersect ${\rm int}(Q_2)$. Then 
$f^n(Q_1)\supset Q_2$, that is $Q_1\supset f^{-n}(Q_2)$. But
$f^{-n}(\partial^u Q_2)$ converges to $\mathcal O_2$, which is disjoint from $Q_1$, a contradiction.
The proof is complete.
\end{proof}

\begin{proposition}\label{prop.hom-rel}
If $\mu$ and $\nu$ are $\vf$--invariant ergodic probability measures with positive entropy,
then $\mu\stackrel{h}{\sim}\nu$.
\end{proposition}

\begin{proof}
This is an adaptation of  \cite[Corollary 1.5]{BCS}. Let $\mu$ and $\nu$ be $\vf$--invariant 
ergodic probability measures with positive entropy.

By Proposition \ref{prop-periodic-point}, there are periodic orbits $\mathcal O_1,\mathcal O_2$ disjoint
from $\partial\Sigma$ such that $\mathcal O_1$ is homoclinically related to $\mu$,
$\mathcal O_2$ is homoclinically related to $\nu$, and $su$--quadrilaterals $Q_1,Q_2$ defined
by these orbits. We can choose these quadrilaterals satisfying the assumptions of Proposition 
\ref{prop-topological-intersection}. Therefore, there is an unstable curve $\gamma:[0,1]\to {\rm int}(\Sigma)$
of $\mathcal O_1$ such that $x=\gamma(0)\not\in Q_2$ and $y=\gamma(1)\in {\rm int}(Q_2)$. 
By the Jordan curve theorem, there is $t\in (0,1)$ such that $\gamma(t)\in \partial^sQ_2$.

Let $\Delta'\subset M$ be a Pesin block with $\mu(\Delta')>0$. Take $\Delta \subset \Delta' $ of full measure inside $\Delta'$ such that every point in $\Delta$ is a density point of $\Delta'$. Let $\delta>0$ be such that every
unstable Pesin manifold of points in $\Delta'$ contains a disk of radius $\delta$; we denote this set
by $W^u_\delta(z)$ for $z\in \Delta'$. The sets $W^u_\delta(z)$ vary continuously with $z\in\Delta'$.

As $\mathcal O_1 \stackrel{h}{\sim}\mu$, by the inclination lemma we can also assume that $W^s(\mathcal O_1)\pitchfork W^u_{\delta/10}(z) \neq \emptyset$ for any  $z\in \Delta'$.  Fix $z\in \Delta$ and
a small disk $D^s \subset W^s(\mathcal{O}_1)$ which contains a transverse intersection of $W^s(\mathcal{O}_1)$ and $W^u_{{\delta/10}}(z)$.  By the continuity of $z\in\Delta'\mapsto W^u_\delta(z)$, if 
$r>0$ is small then $D^s \pitchfork W^u_{\delta}(z') \neq \emptyset$ for every $z'\in B_r(z)\cap \Delta'$. 

Take the disk $D\subset W^u(\mathcal O_1)$,  defined by $D=\varphi^{[-a,a]}(\gamma)$, for some $a>0$,  and observe that $\gamma$ is contained in the transverse intersection of $D$ with $\interior (\Sigma)$.  By the inclination lemma again with $D$ and $W^u_\delta(z)$, there exists $T>0$ such that $\varphi^T(W^u_\delta(z))\cap \interior (\Sigma) \neq \emptyset$, and it contains a continuous curve $\widehat{\gamma}$ which is close to $\gamma$ and, in particular,  it intersects both $\interior(Q_2)$ and $\interior(\Sigma)\setminus Q_2$.
By continuity, for any $z' \in B_r(z) \cap \Delta'$ we have that
$\varphi^T(W^u_{\delta}(z'))\cap {\rm int}(\Sigma)$ contains a curve $\gamma'$ which is also close to $\gamma$, and intersects both  $\interior(Q_2)$ and $\interior(\Sigma)\setminus Q_2$. 

Recall that $z\in \Delta$ is a density point of $\Delta'$. Therefore, we may assume that $\mu(B_r(z))>0$.
We can also assume that $r$ is much smaller than the size of the stable manifolds for points in $\Delta'$.
For $z' \in \Delta' \cap B_r(z)$, let $W^{ss}_r(z')$ be the connected component of $W^{ss}(z')\cap B_r(z)$ that contains $z'$, and let $z'\mapsto \mu^{ss}_{z'}$ be the disintegration of $\mu$ restricted to $\Delta' \cap B_r(z)$
with respect to the partition given by $W^{ss}_r(z')$, $z' \in \Delta' \cap B_r(z)$. In particular,
$\mu^{ss}_{z'}(\Delta') >0$ for $\mu$--a.e. $z' \in \Delta' \cap B_r(z)$.

By the Ledrappier-Young entropy formula \cite{LY2},  we have that $\mu^{ss}_{z'}$ has positive dimension.
Hence, we can take a point $w \in \Delta' \cap B_r(z)$ and a set $\tau\subset W^{ss}_r(w)\cap \Delta'$ with positive Hausdorff dimension. We thus define a lamination $\mathfs L:= \bigcup_{z' \in \tau} W^u_\delta(z')$.
By the above discussion, every leaf of $\mathfs L$ intersects $D^s$ transversely.  By the inclination lemma, $\mathfs L_T:= \varphi^T\left( \mathfs L \right)$ intersects transversely $\interior(\Sigma)$ near $\gamma$.  By the transversality of the flow with the Birkhoff section, we obtain a lamination $\widehat{\mathfs L}$ in $\interior(\Sigma)$ with the following properties:
\begin{enumerate}[$\circ$]
\item $\widehat{\mathfs L}$ is contained in the intersection of $\mathfs L_T$ and $\interior(\Sigma)$.
\item Every leaf of $\widehat{\mathfs L}$ is a $C^\infty$ curve close to $\gamma$, hence
intersects $\interior(Q_2)$ and $\interior(\Sigma) \setminus Q_2$.
\item $\widehat{\mathfs L}$ has positive transverse Hausdorff dimension.
\item The holonomy of $\widehat{\mathfs L}$ is Lipschitz.
\end{enumerate}
The first item is just transversality. The second item is a consequence of the inclination lemma, as 
already explained. The third item holds because $\mathfs L$ has positive Hausdorff dimension.
It remains to justify the fourth item. Since $M$ has dimension three, inside a Pesin block,
the unstable lamination is absolutely continuous. This unstable lamination has codimension one,
a context in which absolute continuity is known to imply Lipschitz holonomies.
Since $\widehat{\mathfs L}$ is the transverse intersection of $\mathfs L_T$, which has Lipchitz holonomies,
with $\interior(\Sigma)$, we conclude that $\widehat{\mathfs L}$ has Lipschitz holonomies as well.

Given a point $y$ in $\widehat{\mathfs L}$, let $\widehat{\mathfs L}(y)$ be its leaf.
Fix $\widehat{\tau}$ a small segment inside $\interior (\Sigma)$ transverse to $\widehat{\mathfs L}$.  By \cite[Theorem~4.2]{BCS}, the set 
$$
\{y\in \widehat{\tau}: \widehat{\mathfs L}(y)\text{ is not transverse to }\partial^sQ_2\}
$$ 
has zero Hausdorff dimension.
Therefore there is $y\in \widehat{\tau}$ such that
$\widehat{\mathfs L}(y)$ is transverse to $\partial^sQ_2$, and so $W^u(y)$ is transverse
to $W^s(\mathcal O_2)$. Let $z_0 \in \Delta$ such that $y\in W^u(z_0)$. Using that $z_0$
is a density point of $\Delta'$ and that the unstable manifolds (of a bounded size) of points in $\Delta'$ 
vary continuously, we conclude that there is a set $A$ of positive $\mu$--measure such
that $W^u(z')\pitchfork W^s(\mathcal O_2)\neq\emptyset$ for $z'\in A$. By ergodicity,
we conclude that $W^u(z')\pitchfork W^s(\mathcal O_2)\neq\emptyset$ for $\mu$--a.e. $z'$.

Interchanging the roles of stable and unstable manifolds, we prove similarly that 
$W^s(z')\pitchfork W^u(\mathcal O_2)\neq\emptyset$ for $\mu$--a.e. $z'$.
Therefore $\mu \stackrel{h}{\sim} \mathcal O_2$ which, by the transitivity of the homoclinic relation,
gives that $\mu \stackrel{h}{\sim} \nu$.
\end{proof}

\begin{proposition}\label{prop-HC}
If $\mu$ is an equilibrium state of a Hölder continuous positive entropy potential,
then ${\rm supp}(\mu)={\rm HC}(\mu)$. In particular, if there is a $\vf$--invariant
hyperbolic measure with full support homoclinically related to $\mu$,
then $\mu$ has full support.
\end{proposition}

\begin{proof}
Let $\mathcal O$ be a hyperbolic periodic orbit homoclinically related to $\mu$. 
Then ${\rm supp}(\mu)\subset {\rm HC}(\mu)={\rm HC}(\mathcal O)$.
Letting $\pi_r:\Sigma_r\to M$ be the coding constructed in \cite{BCL}, we have the following facts,
both proved in \cite[Proof of Theorem 1.2]{LP25}:
\begin{enumerate}[$\circ$]
\item ${\rm supp}(\mu)=\overline{\pi_r(\Sigma_r)}$.
\item $\pi_r(\Sigma_r)$ is dense in ${\rm HC}(\mathcal O)$. 
\end{enumerate}
Hence
$$
\overline{\pi_r(\Sigma_r)}={\rm supp}(\mu)\subset {\rm HC}(\mathcal O)\subset \overline{\pi_r(\Sigma_r)},
$$
which implies ${\rm supp}(\mu)={\rm HC}(\mathcal O)={\rm HC}(\mu)$.
The second part of the statement is obvious.
\end{proof}

\begin{proof}[Proof of Theorem \ref{maintheorem}]
Let $\psi:M \to \mathbb{R}$ be a H\"older continuous positive entropy potential (or more generally, a
potential satisfying Remark \ref{rmk-admissible}). Let $\mu$ and $\nu$ be ergodic equilibrium states of $\psi$. 
Since $\psi$ is a positive entropy potential, $\mu$ and $\nu$ have positive entropy. 
By Proposition~\ref{prop.hom-rel}, we have $\mu\stackrel{h}{\sim}\nu$. By Theorem~\ref{thm-BCL},
it follows that $\mu=\nu$. By \cite{LLS}, $\mu$ is Bernoulli times a rotation, and it is Bernoulli when
$\vf$ is a Reeb flow.

Finally, assume that there exists a $\vf$--invariant probability measure $\eta$ with full support and
positive entropy. Again by Proposition~\ref{prop.hom-rel}, we have $\eta\stackrel{h}{\sim}\mu$,
and so Proposition~\ref{prop-HC} implies that $\mu$ has full support. The proof is now complete.
\end{proof}

\subsection{Uniqueness of equilibrium states for $C^r$ flows}\label{sec.finiteregularity}

The $C^\infty$ regularity of the flow in Theorem \ref{maintheorem} is used in two places:
\begin{enumerate}[$\circ$]
\item To guarantee the existence of equilibrium states for continuous potentials \cite{Newhouse-Entropy};
\item To apply the version of Sard Lemma \cite[Theorem~4.2]{BCS} and then obtain
transverse intersections of certain invariant manifolds.
\end{enumerate}
We can state a weaker version of Theorem \ref{maintheorem} for $C^r$ flows ($r>1$).

\begin{theorem}\label{finitereg}
Let $\vf:M\to M$ be a $C^r$ flow $(r>1)$ with positive speed on a closed
three dimensional $C^\infty$ Riemannian manifold $M$, and assume $\vf$ is transitive
and has a Birkhoff section. If $\psi:M\to\R$ is a potential for which every equilibrium state $\mu$
satisfies
\[
h_\mu(\vf) > \frac{1}{r} \max\left\{ \sup\limits_{x\in M} \|d\vf^1(x)\|, \sup\limits_{x\in M} \|d\vf^{-1}(x)\|\right\},
\]
then $\psi$ has at most one equilibrium state.  
\end{theorem}

\begin{proof}
By the Ledrappier-Young entropy formula, a measure $\mu$ satisfying the above entropy assumption 
has stable/unstable dimension greater than $1/r$, which is enough to apply the version of Sard Lemma \cite[Theorem~4.2]{BCS} and obtain a transverse intersection. From this point,
repeat the arguments in the proof of Theorem \ref{maintheorem}.
\end{proof}

\section{Proof of Theorems~\ref{theoremSRB} and \ref{theorem-NUH}}

Following \cite{BCL}, given $\chi>0$ let $\operatorname{NUH}(\chi)\subset M$
be the non-uniformly hyperbolic locus of $\varphi$,
which is the set of points $x\in M$ having a decomposition
$T_xM=E^s_x\oplus \langle X(x)\rangle\oplus E^u_x$ where vectors in $E^{s/u}_x$ are 
contracted/expanded in a non-uniform way measured by $\chi$, and that satisfy some recurrence
conditions; for the precise definition, see \cite[Section~3.1]{BCL}.
Define the {\em geometric potential} $\psi^u:M\to \mathbb{R}$ as 
$$\psi^u(x)=\left\{\begin{array}{ll}
-\lim\frac{1}{t}\log \|d\varphi^t(x)|_{E^u_x}\| & ,x\in \operatorname{NUH}(\chi)\\ 
0&, \textrm{otherwise}. 
\end{array}\right.
$$
The above definitions are not canonical, since they depend on $\chi$. 
Nevertheless, by the Ruelle inequality and the Pesin entropy formula, every SRB measure for $\vf$ is 
an equilibrium state of $\psi^u$ for $\chi>0$ small enough.

Let $\mu,\nu$ be ergodic SRB measures for $\varphi$. By the Pesin entropy formula,
$\mu$ and $\nu$ have positive entropy. The proof of Theorem~\ref{theoremSRB} consists on 
showing that $\mu$ and $\nu$ are homoclinically related. For that, we prove that
the proof of Proposition~\ref{prop.hom-rel} applies. Recall we are assuming $\vf$
is a $C^r$ flow ($r>1$). 

\begin{claim} In the above context, Proposition~\ref{prop.hom-rel} holds true for ergodic $SRB$ measures.\end{claim}

\begin{proof}[Proof of Claim]
Following the same construction of Proposition~\ref{prop.hom-rel} applied to stable manifolds,
we obtain a lamination $\widehat{\mathfs{L}}$ in $\interior (\Sigma)$ whose leaves are
pieces of stable manifolds. This is a lamination with $C^r$ leaves and Lipschitz holonomies. By 
\cite[Theorem~4.2]{BCS}, the Hausdorff dimension of the set of 
$$
\{y\in \widehat{\tau}: \widehat{\mathfs L}(y)\text{ is not transverse to }\partial^uQ_2\}
$$ 
is at most $1/r$. Using that $\mu$ is an SRB measure, its disintegration on unstable
manifolds is absolutely continuous, hence we can take $\widehat\tau$ so that
$\widehat\tau\cap\widehat{\mathfs L}$ has Hausdorff dimension equal to one (in particular, greater
than $1/r$). The rest of the argument follows ipsis literis the proof of Proposition~\ref{prop.hom-rel}.
The conclusion is that for $\mu$--a.e. $x\in M$ and $\nu$--a.e. $y\in M$ we have
$W^s(x)\pitchfork W^u(y)\neq \emptyset$. Interchanging the roles of $\mu$ and $\nu$, 
it follows that $\mu \stackrel{h}\sim \nu$.
\end{proof}

Now we conclude the proof of Theorem~\ref{theoremSRB}.
Choose $\chi>0$ small enough so that $\mu,\nu$ are equilibrium state of $\psi^u$.
As explained in Remark~\ref{rmk-admissible}, Theorem~\ref{thm-BCL} holds for
$\psi^u$, since the lift of this potential to the symbolic space given by Theorem~\ref{thm-BCL}
is cohomologous to a Hölder continuous potential, and so $\mu=\nu$.

The ergodicity part of Theorem~\ref{theorem-NUH} follows from Theorem~\ref{theoremSRB} and 
the fact that ergodic components of the Liouville measure are SRB measures whenever the
Liouville measure is  hyperbolic. The Bernoulliness of the measure follows from the main result in \cite{LLS}.

\bibliographystyle{alpha}
\bibliography{bibliography}{}

\end{document}